\documentclass[11pt]{amsart}

\usepackage[T1]{fontenc}
\usepackage[utf8]{inputenc}
\usepackage{textcomp}
\usepackage{lmodern}
\usepackage{microtype}

\usepackage{amssymb}
\usepackage{amsthm}
\usepackage{amscd}

\usepackage{hyperref}

\usepackage{mathscinet}
\usepackage[giveninits=true,doi=false,url=false,sortcites,backend=biber,backref=true]{biblatex}
\newtheorem{theorem}{Theorem}[section]

\newtheorem{lemma}[theorem]{Lemma}

\theoremstyle{definition}
\newtheorem{definition}[theorem]{Definition}

\allowdisplaybreaks[2]

\begin{document}

\title{An Aldous--Hoover Theorem for Radon Distributions}
\author{Henry Towsner}
\date{\today}
\thanks{Partially supported by NSF grant DMS-2054379}
\address {Department of Mathematics, University of Pennsylvania, 209 South 33rd Street, Philadelphia, PA 19104-6395, USA}
\email{htowsner@math.upenn.edu}
\urladdr{\url{http://www.math.upenn.edu/~htowsner}}

\begin{abstract}
We show that the Aldous--Hoover Theorem, giving representations for exchangeable arrays of Borel-valued random variables, extends to random variables where the common distribution of the random variables is Radon, or even merely compact, a weaker condition that does not even require that the values come from a Hausdorff space. This extends work of Alam \cite{alam2023generalizing} who showed a similar generalization of the di Finetti--Hewitt-Savage Theorem.
\end{abstract}

\maketitle

\section{Introduction}

Consider an infinite sequence of $\{0,1\}$-valued random variables $\{\mathbf{X}_i\}_{i\in\mathbb{N}}$ which are not only identically distributed, but whenever we take $k$ of these random variables, the probability that they are simultaneously equal to $1$ depends only on the number $k$, and not the particular random variables involved---$\mathbb{P}(\mathbf{X}_1=1 \text{ and }\mathbf{X}_2=1)$ is the same as $\mathbb{P}(\mathbf{X}_3=1 \text{ and }\mathbf{X}_{17}=1)$, and similarly for $k$ variables instead of $2$. This immediately implies that the probability of any event involving a combination of the random variables remains the same if we permute the indices.  Such a sequence is called \emph{exchangeable}.

If the $\mathbf{X}_i$ are independent as well as identically distributed then they are certainly exchangeable. However there are other exchangeable sequence: suppose we have two coins, one biased towards heads and one towards tails, and we first randomly choose a coin and then determine the $\mathbf{X}_i$ by flipping the chosen coin once for each $i$ and recording a $1$ if the coin is heads and a $0$ if the coin is tails. Then the $\mathbf{X}_i$ will not be independent---knowing that $\mathbf{X}_1=1$ makes it more likely that the coin was the one biased towards heads, which makes it more likely that $\mathbf{X}_2=1$ as well---but they remain exchangeable.

More generally, if $\mu$ is any distribution on $[0,1]$, we obtain an exchangeable sequence $\mathbf{Y}^\mu$ by first choosing $p$ according to $\mu$ and then letting each $\mathbf{Y}^\mu_i=1$ independently with probability $p$. Di Finetti's Theorem \cite{diFinetti,alvarezmelis2015translation} says that all exchangeable sequences have this form: for any exchangeable sequence $\{\mathbf{X}_i\}$, there is a distribution $\mu$ on $p$ so that random variables $\{\mathbf{X}_i\}$ have the same distribution as $\mathbf{Y}^\mu$.

This result has been extended to random variables taking values in larger spaces than $\{0,1\}$---the Borel measure on the unit interval \cite{MR1508036,MR0055601}, the Baire $\sigma$-algebra on a compact Hausdorff space \cite{MR76206}, and, as a consequence, the Borel $\sigma$-algebra on any analytic subset of a Polish space \cite{MR159923}. However Dubins and and Freedman constructed a space and a sequence of random variables \cite{Dubins:1979aa} which do not satisfy the conclusion of di Finetti's Theorem, showing that some restrictions on the space of values are required. Recently, Alam showed \cite{alam2023generalizing} an extension to random variables using only the assumption that the common distribution (the distribution of an individual random variable $\mathbf{X}_i$) is Radon (indeed, slightly less than this).

In a different direction, the Aldous--Hoover Theorem \cite{aldous:MR637937,hoover:arrays} generalizes di Finetti's Theorem to exchangeable \emph{arrays}. By an array of random variables, we mean a collection of random variables $\{\mathbf{X}_e\}_{e\in\binom{\mathbb{N}}{n}}$ indexed by sets of integers of size exactly $n$. This array is exchangeable if the joint probability distribution is preserved under permutations of the indices---that is, the probability of an event involving several of the events in the array is preserved if we replace each $\mathbf{X}_e$ with $\mathbf{X}_{\pi(e)}$ where $\pi$ is some permutation of $\mathbb{N}$. Stating this more formally becomes notationally heavy, so we define this precisely below. (There are many variations on this notion---where the variables are indexed by finite sequences instead of finite sets, or by more complicated objects like finite trees, or with various restrictions on the permutations considered \cite{kallenberg:MR2161313,MR3230009,2015arXiv150906733C,2015arXiv150906170A,MR4213163}. Since we are concerned in this paper with the range of the random variables, and not the exact amount of symmetry the array has, we consider only these arrays indexed by finite sets.)

To state the Aldous--Hoover Theorem, it helps to think of di Finetti's Theorem as decomposing the random information in the $\mathbf{X}_i$ into a global part shared by all the $\mathbf{X}_i$ (the choice of a particular value of $p$ according to $\mu$) and a local part specific to each $\mathbf{X}_i$ and independent from all other variables in the sequence. That is, we can think of di Finetti's Theorem as saying that we can write any exchangeable array in the form
\[\mathbf{Y}^\rho_i=\rho(\xi_\emptyset,\xi_i)\]
where $\xi_\emptyset$ and the $\{\xi_i\}_{i\in\mathbb{N}}$ are chosen i.i.d.\ from a suitable distribution and $\rho$ is a measurable function.

The analogous decomposition for an array requires us to have not only a global part and a local part, but also intermediate parts shared by variables with some common coordinates. We choose random variables $\{\xi_s\}$ i.i.d.\ for every set $s\subseteq\mathbb{N}$ of size at most $n$, and then set $\mathbf{Y}^\rho_e=\rho(\{\xi_s\}_{s\subseteq e})$.

The Aldous--Hoover Theorem says that every exchangeable array of random variables valued in a standard Borel space has the same distribution of one of the form $\mathbf{Y}^\rho_e$. (Indeed, since all standard Borel spaces are isomorphic, the $\xi_s$ can be chosen to come from $[0,1]$ with the usual Lebesgue measure.)


In this paper we give a short proof of a common generalization, showing that the Aldous--Hoover Theorem holds whenever the common distribution of the random variables is \emph{compact} (inner regular for some compact collection of sets), a property close to the one used in \cite{alam2023generalizing}, but which does not even require the underlying space to be Hausdorff. Indeed, the proof follows almost immediately by combining two existing ingredients---Hoover's original proof using nonstandard analysis \cite{hoover:arrays} and Ross' work extending the standard part map to compact spaces \cite{MR1079898}.

\section{Exchangeable Arrays}

In order to talk about the distribution of an exchangeable array of random variables, it will be helpful to introduce some notation. Let $\{\mathbf{X}_{\{i_1,\ldots,i_n\}}\}_{i_1,\ldots,i_n\in\mathbb{N}}$ be an array of random variables valued in some topological space $T$ and let $\mathcal{B}$ be the $\sigma$-algebra generated by the open sets on $T$. (Initially it is convenient to think of $T$ as $[0,1]$ with the usual topology, so $\mathcal{B}$ is the Lebesgue measurable sets, but we are ultimately interested in more general cases.)


The basic events which describe the distribution of such an array are given by hypergraphs:
\begin{definition}
  An \emph{event hypergraph} (in $\mathcal{B}$) consists of a finite set $S$ and, for each $e\in{S\choose n}$, a measurable set $B_e\in \mathcal{B}$. When $(S,\{B_e\})$ is an event hypergraph, $f:S\rightarrow\mathbb{N}$ is a function  $\{\mathbf{X}_{\{i_1,\ldots,i_n\}}\}_{i_1,\ldots,i_n\in\mathbb{N}}$ is an array of random variables, we write $B_{S,\{B_e\},f,\mathbf{X}}$ for the event
  \begin{quote}
    for every $e\in{S\choose n}$, $\mathbf{X}_{f(e)}\in B_e$.
  \end{quote}

$\{\mathbf{X}_{\{i_1,\ldots,i_n\}}\}_{i_1,\ldots,i_n\in\mathbb{N}}$ is \emph{exchangeable} if the distribution of $B_{S,\{B_e\},f,\mathbf{X}}$ is the same for all event hypergraphs $(S,\{B_e\})$ and all injective functions $f$. In this case the \emph{common distribution} of $\{\mathbf{X}_{\{i_1,\ldots,i_n\}}\}_{i_1,\ldots,i_n\in\mathbb{N}}$ is the distribution of $\mathbf{X}_{\{1,\ldots,n\}}$ (by exchangeability, the particular choice of indices does not matter).
\end{definition}


When $\{\mathbf{X}_{i_1,\ldots,i_n}\}_{i_1,\ldots,i_n\in\mathbb{N}}$ is exchangeable, we will often write simply $\mathbb{P}(B_{S,\{B_e\},\mathbf{X}})$ in this case, omitting $f$.

We will always consider event hypergraphs for a single fixed value $n$, the arity of the exchangeable random variable we are considering. In the unary case, an ``event hypergraph'' is simply a finite collection of measurable sets $\{B_s\}_{s\in S}$, and the event $B_{S,\{B_s\},f}$ is essentially the event $\prod_s B_s$.

\section{Motivation}

\subsection{Proving Aldous--Hoover}
Our starting point is to examine the outline of the usual proof of the Aldous--Hoover Theorem for a $\mathbb{R}$-valued random variable using nonstandard analysis. For simplicity, let us focus on the case where the array is dissociated---where $B_{S,\{B_s\},f,\mathbf{X}}$ and $B_{S',\{B'_s\},f',\mathbf{X}}$ are independent whenever $f$ and $f'$ have disjoint range. Let $\{\mathbf{X}_{i_1,\ldots,i_n}\}_{i_1,\ldots,i_n\in\mathbb{N}}$ be a dissociated exchangeable array.

In this case, the Aldous--Hoover Theorem tells us there is a Borel measurable function $\rho$ so that if we choose random variables $\{\xi_\sigma\}_{\sigma\subseteq\mathbb{N}, 0<|\sigma|\leq n}$ i.i.d. then the array
\[\mathbf{Y}_{i_1,\ldots,i_n}=\rho(\{\xi_\sigma\}_{\emptyset\subsetneq\sigma\subseteq\{i_1,\ldots,i_n\}})\]
has the same distribution as $\mathbf{X}_{i_1,\ldots,i_n}$. Being dissociated is equivalent to not needing the $\xi_\emptyset$ datum in this representation.

To construct $\rho$, we first choose a random sample from $\{\mathbf{X}_{i_1,\ldots,i_n}\}_{i_1,\ldots,i_n\in\mathbb{N}}$---that is, we choose some $\omega$ and consider the array of random values $\{\mathbf{X}_{i_1,\ldots,i_n}(\omega)\}_{i_1,\ldots,i_n\in\mathbb{N}}$. We consider the corresponding finite arrays of values $M_k=\{\mathbf{X}_{i_1,\ldots,i_n}(\omega)\}_{i_1,\ldots,i_n\leq k}$.

Since the arrays $M_k$ are finite, we can think of sampling from them. In particular, when $(S,\{B_e\})$ is an event hypergraph, we associate an event $B_{S,\{B_e\},M_k}$ as follows: for each $s\in S$, choose $i_s\leq k$ uniformly at random, and consider the event that the $i_s$ are all distinct and, for each $e$, $M_k[\{i_s\}_{s\in e}]\in B_e$.

For a fixed $(S,\{B_e\})$, as $k$ gets large, the probability that the $i_s$ are pairwise distinct approaches $1$, and a standard second moment argument (using the fact that $\mathbf{X}$ is dissociated) shows that, with probability $1$,
\[\lim_{k\rightarrow\infty}\mathbb{P}(B_{S,\{B_e\},M_k})=\mathbb{P}(\{\mathbf{X}_{i_1,\ldots,i_n}\}_{i_1,\ldots,i_n\in\mathbb{N}}\in B_{S,\{B_e\}}).\]
That is, the distribution we get by sampling from the array $M_k$ approaches the distribution given by the original exchangeable array.

Nonstandard analysis gives us a hyperfinite array $M^*=\{x_{i_1,\ldots,x_n}\}_{i_1,\ldots,x_n\leq k^*}$ of values in the nonstandard reals, $\mathbb{R}^*$. The space $[0,k^*]^n$ has a probability measure on it (the Loeb measure corresponding to the usual counting measure on a finite set). Then hypergraph event $(S,\{B_e\})$ no longer makes sense because $M^*$ is not $\mathbb{R}$-valued, but we can consider the associated event $(S,\{B_e^*\})$. We can choose $i_s\in[0,k^*]$ for each $s\in S$ according to this probability measure, and consider the event $B_{S,\{B_e\},M^*}$ that, for each $e$, $M^*[\{i_s\}]\in B_e^*$. By the usual properties of nonstandard analysis, we have
\[\mathbb{P}(B_{S,\{B_e\},M^*})=\lim_{k\rightarrow\infty}\mathbb{P}(B_{S,\{B_e\},M_k})=\mathbb{P}(B_{S,\{B_e\},\mathbf{X}}).\]

The Loeb measure gives a measure space on $[0,k^*]^n$, but not a product of standard Borel spaces: it contains subsets of $[0,k^*]^n$ which are not generated by boxes, and therefore would not be present in the product space. However there are spaces $\{\Omega_s\}_{0<s\leq n}$ and a measurable, measure-preserving function $\pi:[0,k^*]^n\rightarrow \prod_s\Omega_{|s|}$ so that $\pi$ induces an isomorphism of measure algebras\footnote{This is the most technical step, and the main role of nonstandard analysis in the proof is to make the construction of this decomposition easier. The details are presented below.}. Since it suffices to consider countably many events, we can restrict to a separable subspace of each $\Omega_s$ and then use Maharam's lemma to replace each with the unit interval.

Putting this together, for each event hypergraph $(S,\{B_e\})$, we have, up to measure $0$, an event $D_{S,\{B_e\}}\subseteq\prod_s\Omega_{|s|}$ given by the image of $B_{S,\{B_e\},M^*}$ and so that $\mathbb{P}(D_{S,\{B_e\}})=\mathbb{P}(B_{S,\{B_e\},M^*})=\mathbb{P}(B_{S,\{B_e\},\mathbf{X}})$. 

The construction of $\pi$ ensures that $\pi(i_1,\ldots,i_n)_s$ depends only on $\{i_j\}_{j\in s}$. This has the crucial consequence that the events $D_{S,\{B_e\}}$ respect intersections: the event $D_{S,\{B_e\}}$ is the event
\[\{\{\omega_s\}_{\emptyset\neq s\subseteq s}\mid \forall e\in{S\choose n}\, \{\omega_s\}_{\emptyset\neq s\subseteq e}\in D_{\{1,\ldots,n\},B_e}\}.\]




So we have reproduced most of the probabilistic structure on the events $B_{S,\{B_e\},\mathbf{X}}$, but we have lost the random variable, and need to construct a new one. For any measurable $B\subseteq\mathbb{R}$, consider the hypergraph event $\hat I=(\{1,\ldots,n\},I)$. We have a corresponding event $D_{\hat I}$ (up to measure $0$) in $\prod_s\Omega_s$, and we could hope to define $\rho:\prod_s\Omega_s\rightarrow\mathbb{R}$ so that $\rho^{-1}(I)=D_{\hat I}$.

Because of the measure $0$ errors, we cannot do this for all sets simultaneously, but we can focus on the closed intervals: for each $\{\omega_s\}\in\prod_s\Omega_s$, we consider all those intervals $[a,b]$ so that $\{\omega_s\}\in D_{\widehat{[a,b]}}$. By compactness, the intersection is non-empty, so we may choose $\rho(\{\omega_s\})$ to be an arbitrary element of this intersection. (Indeed, for almost every point, this element is uniquely determined, but it turns out we don't need this, and may not have this in the fully general setting.) For a general set $B\subseteq\mathbb{R}$, we now have two corresponding subsets of $\prod_s\Omega_s$---$\rho^{-1}(B)$ and $D_{\hat B}$, and in general these are not the same. They do, however, have the same measure, since they agree when $B$ is an interval, and the intervals generate $\mathcal{B}$.

We can use $\rho$ to define an array of random variables: for each finite subset $t$ of $\mathbb{N}$ with $0<|t|\leq n$, we choose $\omega_t\in\Omega_{|t|}$ randomly according to the measure on $\Omega_{|t|}$, and we then set $\mathbf{Y}^\rho_{i_1,\ldots,i_n}=\rho(\{\omega_t\}_{\emptyset\neq t\subseteq\{i_1,\ldots,i_n\}})$. By definition, for any $B$ we have $\mathbb{P}(B_{\{1,\ldots,n\},B,\mathbf{Y}^\rho})=\mu(\rho^{-1}(B))=\mathbb{P}(D_{\{1,\ldots,n\},B})$. Because the events $D_{S,\{B_e\}}$ respect intersections, we also have $\mathbb{P}(B_{S,\{B_e\},\mathbf{Y}^\rho})=\mathbb{P}(D_{S,\{B_e\}})$, and therefore $\mathbf{Y}^\rho$ gives our representation of $\mathbf{X}$.

\subsection{Counterexamples}

Next, we consider the example of a space that di Finetti's Theorem does not apply to, to see what the obstacle is. The argument in \cite{Dubins:1979aa} constructs a set $S\subseteq[0,1]$ using a transfinite induction to ensure that the collection of infinite sequences from $S$ has Lebesgue outer measure $1$. This means that we can take any $\mathbb{R}$-valued exchangeable sequence $\{\mathbf{X}_i\}$ and restrict it to $S$, obtaining an $S$-valued exchangeable sequence $\{\mathbf{Y}_i\}$. They then construct an $\mathbb{R}$-valued exchangeable sequence $\{\mathbf{X}_i\}$ in such a way that the restricted $S$-valued sequence $\{\mathbf{Y}_i\}$ cannot have a decomposition of the kind promised by di Finetti's Theorem.

The first part of the proof of Aldous--Hoover outlined above goes through without change: we can sample a particular sequence $\{\mathbf{Y}_i(\omega)\}$ of values from $S$. Moreover, because $S$ was constructed so that the sequences from $S$ have Lebesgue outer measure $1$, this sequence ``looks just like'' the sequence we would have gotten if we had instead started with the original, $\mathbb{R}$-valued, sequence $\{\mathbf{X}_i\}$.

Indeed, in some sense the nonstandard arguments ``can't tell'' that we are supposed to be working in the restriction to $S$. The array $M^*$ looks just like an array we would get working with $\mathbb{R}$-valued random variables, so while the decomposition part of the argument goes through unchanged, efforts to get a random variable with values in some standard space would naturally land us in $\mathbb{R}$, not in $S$.

\subsection{Compact Measures}

In order to avoid this obstacle, we need our topological space to have enough structure to mimic the standard part map. As it happens, the topological properties needed to produce such a map have been studied. The most general result of this kind I know of is in \cite{MR1079898}, and we use the argument from there to complete our proof.

\begin{definition}
  Let $(X,\mathcal{B},\mu)$ be a probability space. We say this space is \emph{compact} if there exists a collection $\mathcal{K}\subseteq\mathcal{B}$ such that:
  \begin{itemize}
  \item the collection $\mathcal{A}$ is compact---whenever $\mathcal{A}\subseteq\mathcal{K}$ has the finite intersection property (every finite subset of $\mathcal{A}$ has non-empty intersection), $\bigcap\mathcal{A}\neq\emptyset$ as well, and
  \item $\mu$ is $\mathcal{K}$-inner regular---for every set $B\in\mathcal{B}$, $\mu(B)=\sup_{K\subseteq B, K\in\mathcal{K}}\mu(K)$.
  \end{itemize}
\end{definition}

This notion is a slight generalization of the usual notion of a Radon measure: when $X$ is a Hausdorff topological space and $\mathcal{B}$ is the $\sigma$-algebra generated by the open sets, the measure space is Radon precisely when it is compact in the sense above with $\mathcal{K}$ the collection of compact sets. (Ross shows \cite{MR1079898} that any compact measure space has an extension which is a Radon space.)


\section{Main Theorem}






\begin{theorem}
  Suppose $\{\mathbf{X}_{\{i_1,\ldots,i_n\}}\}_{i_1,\ldots,i_n\in\mathbb{N}}$ is an exchangeable array of random variables valued in a measurable space $(T,\mathcal{U})$ such that the common distribution of the individual random variable $\mathbf{X}_{\{i_1,\ldots,i_n\}}$ is compact, and let $\mathcal{K}\subseteq\mathcal{U}$ the collection of sets witnessing this.

  Then there are probability measure spaces $\{(\Delta_k,\mathcal{D}_k,\mu_k)\}_{k\leq n}$ and a measurable function $\rho:\prod_{\sigma\subseteq[n]}\Delta_{|\sigma|}\rightarrow T$ so if $\{\xi_\sigma\}_{\sigma\in\binom{\mathbb{N}}{\leq n}}$ are chosen uniformly at random from $\Delta_{|\sigma|}$ then the exchangeable random variable given by $\mathbf{Y}_{\{i_1,\ldots,i_n\}}=\rho(\{\xi_\sigma\}_{\sigma\subseteq\{i_1,\ldots,i_n\}})$ has the same distribution as $\mathbf{X}$.
\end{theorem}
\begin{proof}
The first part of the proof is essentially the nonstandard proof of Aldous--Hoover. Since Hoover's original work has not been published, we include a brief but complete version of this argument here. 
  
  Let $(\Omega,\mathcal{B},\mu)$ be the sample space for the random variables $\mathbf{X}_{\{i_1,\ldots,i_n\}}$, so we may view our array as a collection $\{t_{\omega,\{i_1,\ldots,i_n\}}\}_{(\omega,i_1,\ldots,i_n)\in\Omega\times(\mathbb{N})^n}$ which is symmetric under permutations of the coordinates from $\mathbb{N}$.  For any $M$ we may define a new probability measure on the event hypergraphs: given $(S,\{B_e\})$, we choose $\omega\in\Omega$ according to $\mu$ and choose $f:S\rightarrow[0,M]$ uniformly at random. We define $B_{S,\{B_e\},M}$ to be the event that $f$ is injective on $S$ and, for all $e\in{S\choose n}$, $t_{\omega,\{f(i)\}_{i\in e}}\in B_e$.

  For each $(S,\{B_e\})$, $\lim_{M\rightarrow\infty}\mathbb{P}(B_{S,\{B_e\},M})=\mathbb{P}(B_{S,\{B_e\},\mathbf{X}})$: if $M$ is large enough that $S$ is almost certainly injective then $\mathbb{P}(B_{S,\{B_e\},M})$ is approximately the average over all injections $f:S\rightarrow[M]$ of $B_{S,\{B_e\},f,\mathbf{X}}$.  But for every $f$, $\mathbb{P}(B_{S,\{B_e\},f,\mathbf{X}})=\mathbb{P}(B_{S,\{B_e\},\mathbf{X}})$, so in the limit, $\mathbb{P}(B_{S,\{B_e\},M})$ approaches $\mathbb{P}(B_{S,\{B_e\},\mathbf{X}})$. 

  We fix a suitable nonstandard extension (for instance, coming from an ultraproduct), so we obtain an array $\{t_{\omega,i_1,\ldots,i_n}\}_{(\omega,i_1,\ldots,i_n)\in\Omega^*\times(\mathbb{N}^*)^{n}}$ valued in $T^*$. Fix an element $M^*\in\mathbb{N}^*\setminus\mathbb{N}$ and restrict to the array $\{t_{\omega,i_1,\ldots,i_n}\}_{(\omega,i_1,\ldots,i_n)\in\Omega^*\times(M^*)^{n}}$ where $[M^*]=\{i\in\mathbb{N}^*\mid i<M^*\}$. We can view these as arrays indexed by $\Omega^*\times(M^*)^n$ in the usual way, 

  The Loeb measure gives us a collection of probability measures on $\Omega^*$, $[M^*]$, and their products such as $\Omega^*\times[M^*]^n$. We write $\mathcal{B}_S$ for the measurable subsets of $\Omega^*\times[M^*]^S$, or just $\mathcal{B}_k$ when $S=[k]$.   For any event hypergraph $(S,\{B_e\})$, we let $B_{S,\{B_e\},M^*}$ be the event on $\Omega^*\times[M^*]^S$ containing those tuples $\omega,\{i_j\}_{j\in S}$ which holds exactly when, for each $e\in{S\choose n}$, $t_{\omega,\{i_s\}_{s\in e}}\in B^*_e$.  Nonstandard analysis ensures that $\mathbb{P}(B_{S,\{B_e\},M^*})=\lim_{M\rightarrow\infty}\mathbb{P}(B_{S,\{B_e\},M})=\mathbb{P}(B_{S,\{B_e\},\mathbf{X}})$.

  
  We now need to construct our decomposition. Fix some $k\leq n$ and the set $\Omega^*\times[M^*]^k$ and the measurable sets, $\mathcal{B}_k$. For any $s\subseteq[k]$, we define $\mathcal{B}_{k,s}$ to be the sub-$\sigma$-algebra of $\mathcal{B}_k$ containing only sets of the form $\{(\omega,i_1,\ldots,i_k)\mid (\omega,\{i_j\}_{j\in s})\in B\}$ where $B$ is a measurable subset of $\Omega^*\times[M^*]^s$. We define $\mathcal{B}_{k,k-1}$ to be the sub-$\sigma$-algebra of $\mathcal{B}_k$ containing $\mathcal{B}_{k,s}$ for all $s$ with $|s|=k-1$.

  We define $\mathcal{B}^-_k$ to contain all sets which are independent from $\mathcal{B}_{k,k-1}$. Note that $\mathcal{B}_k^-\cup\mathcal{B}_{k,k-1}$ generates $\mathcal{B}_k$: if $B\in\mathcal{B}_k$, it has a projection $f_B^+=\mathbb{E}(B\mid\mathcal{B}_{k,k-1})$ which is $\mathcal{B}_{k,k-1}$-measurable, while the complement $f^-_B=\chi_B-f_B^+$ is $\mathcal{B}_k^-$-measurable.

  We will take our spaces $(\Delta_k,\mathcal{D}_k)$ to be $(\Omega^*\times[M^*]^k,\mathcal{B}^-_k)$. We can define a function $\eta:\Omega^*\times[M^*]^n\rightarrow\Delta=\prod_{s\subseteq[n]}\Delta_{|s|}$ by mapping $(\omega,i_1,\ldots,i_n)$ to $\{(\omega,i_{j_1},\ldots,i_{j_k})\}_{\{j_1,\ldots,j_k\}\subseteq[n]}$. This map is quite far from being injective, since the coordinates are duplicated: the range of $\eta$ consists of those points $\{(\omega_s,\{i_{s,j}\}_{j\in s})\}_{s\subseteq[n]}$ such that $i_{s,j}=i_{t,j}$ whenever $j\in s\cap t$ and $\omega_s=\omega_t$ for all $s,t$.

  Nevertheless, this map gives an isomorphism of measure algebras. To see that it is measurable, it suffices just to consider a single $s$ since the $\sigma$-algebra on $\Delta$ is a product space. So consider some measurable $B\in\mathcal{D}_k$ at the coordinate $s=\{j_1,\ldots,j_k\}$; then $\eta^{-1}(B\times\prod_{t\neq s}\Delta_{|t|})$ is precisely those $(\omega,i_1,\ldots,i_n)$ such that $(\omega,i_{j_1},\ldots,i_{j_k})\in B$. This is a measurable subset---indeed, it is a measurable subset belonging to $\mathcal{B}_{n,s}$. Furthermore, it is what we might call an element of $\mathcal{B}_{n,s}^-$: $B$ belongs to $\mathcal{B}_{n,s}$ but is independent from $\bigcup_{t\subseteq s}\mathcal{B}_{n,t}$.

  It is this fact which will let us show that $\eta$ is measure-preserving. Again, since we are dealing with a product measure, it suffices to show that $\eta$ preserves the measure of any box $\prod_{s\subseteq[n]}B_s$ where each $B_s$ is a measurable subset of $\Delta_{|s|}$. The inverse image $\eta^{-1}(\prod_{s\subseteq[n]}B_s)$ is the intersection $\bigcap_{s\subseteq[n]}\eta^{-1}(B_s)$.

  We claim that, when $B_s\in\mathcal{B}_{n,s}^-$ for all $s\subseteq[n]$, $\mu(\bigcap_s B_s)=\prod_s\mu(B_s)$. To see this, we proceed inductively: let $U$ be a collection of subsets of $[n]$; we claim there is a $t\in U$ so that $\mu(\bigcap_{s\in U}B_s)=\mu(B_t)\mu(\bigcap_{s\in U\setminus\{t\}}B_s)$. Choose $t$ to be any maximal element of $U$. Then
  \begin{align*}
    \mu(\bigcap_{s\in U}B_s)
    &=\int \prod_s\chi_{B_s}(\{i_j\}_{j\in s})\,d\mu(i_1,\ldots,i_n)\\
    &=\int \chi_{B_t}(\{i_j\}_{j\in t}) \int \prod_{s\neq t}\chi_{B_s}(\{i_j\}_{j\in s})\,d\mu\,d\mu(\{i_j\}_{j\in[n]\setminus t}).
  \end{align*}
  Here we use the essential fact that the Loeb measures in the nonstandard extension satisfy the conclusion of Fubini's Theorem, allowing us to freely rearrange the order in which we integrate coordinates.  Because we chose $t\in U$ maximal, for every fixed choice of $\{i_s\}_{s\in[n]\setminus t}$, $\prod_{s\neq t}\chi_{B_s}(\{i_j\}_{j\in s})$ is a product of functions depending on a proper subset of the coordinates in $t$, and therefore belongs to $\bigcup_{s\subsetneq t}\mathcal{B}_{n,s}$, and therefore $B_t$ is independent of this function, so $\mu(\bigcap_{s\in U}B_s)=\mu(B_t)\mu(\bigcap_{s\in U\setminus\{t\}}B_s)$. Applying this repeatedly, we get that $\mu(\bigcap_s B_s)=\prod_s\mu(B_s)$, and therefore $\mu(\eta^{-1}(\prod_{s\subseteq[n]}B_s)=\prod_s\mu(\eta^{-1}(B_s))=\prod_s\mu(B_s)$ as needed.

  Even though $\eta$ is not injective, it is still an isomorphism of measure algebras: for every event $B$ in $\mathcal{B}_n$, there is an event $B'$ in $\Delta$ so that $\eta^{-1}(B')$ differs from $B$ on a set of measure $0$.  To see this, note that $\mathcal{B}_n$ is generated by $\bigcup_s\mathcal{B}^-_{n,s}$, so it suffices to consider events of the form $\bigcap_s B_s$ with each $B_s\in\mathcal{B}^-_{n,s}$, since these generate $\mathcal{B}_n$, and we have already seen that $\bigcap_s B_s=\eta^{-1}(\prod_s B_s)$ where we view each $B_s$ as a measurable subset of $\Delta_{|s|}$.

  For any finite set $S$, we can similarly map $\Omega^*\times[M^*]^S$ to $\prod_{s\in\binom{S}{\leq n}}\Delta_{|s|}$ by mapping $(\omega,\{i_j\}_{j\in S})$ to $\{(\omega,\{i_j\}_{j\in s})\}_{s\in\binom{S}{\leq n}}$. Whenever $e\in\binom{S}{n}$, we have a copy of $\Delta$. (Indeed, we could map $\Delta$ to $\prod_{s\subseteq e}\Delta_{|s|}$ in $n!$ different ways; the symmetry of our hypergraph events ensures that the choice does not matter.)

  In particular, for each $B\in\mathcal{U}$, there is an event $B^*$ in $\mathcal{B}_n$ which is mapped by $\eta$ to an event $\hat B$ in $\Delta$. Crucially, the event $B_{S,\{B_e\},M^*}$ is mapped to the intersection of the corresponding events $\hat B_e$ on the respective images of $\Delta$.


  What remains is to get a random variable with $\Delta$ as our state space. For each $K\in\mathcal{K}$, we have the image $\hat K$ in $\Delta$. $\hat K$ is only determined up to measure $0$, so we pick some representative arbitrarily. We define $\rho:\Delta\rightarrow T$ by setting $\rho(\{\xi_s\})$ to be an arbitrary element of $\bigcap_{K\in\mathcal{K},\{\xi_s\}\in\hat K}K$. (If there are no such $K$, this is an intersection over no sets, and therefore the whole space, so $\rho(\{\lambda_\sigma\})$ may be any point.)

  Let $\mu$ be the measure on $T$ given by the common distribution of the $\mathbf{X}_{i_1,\ldots,i_n}$.  Despite the arbitrary choices involved, we claim that $\rho$ is measurable and measure-preserving. First, note that $\rho^{-1}(K)=\hat K$ for $K\in\mathcal{K}$, so $\mu(\rho^{-1}(K))=\mu(K)$.  Let $B$ be any measurable set.  Then, by $\mathcal{K}$-inner regularity, for any $\epsilon>0$ we may choose $K^+\subseteq B$ and $K^-\subseteq T\setminus B$ from $\mathcal{K}$ so that $\mu(K^+\cup K^-)>1-\epsilon$, and therefore $\mu(\hat K^+\cup \hat K^-)>1-\epsilon$. We have $\hat K^+\subseteq\rho^{-1}(B)$ and $\rho^{-1}(B)\cap \hat K^-=\emptyset$, so $\hat K^+$ approximates $\rho^{-1}(B)$ to within $\epsilon$. Since this holds for all $\epsilon$, $\rho^{-1}(B)$ is a measurable set with measure $\mu(B)$.

  We now consider the exchangeable random variable $\mathbf{Y}$ as defined in the statement of the theorem: we choose $\{\xi_\sigma\}_{\sigma\in\binom{\mathbb{N}}{\leq n}}$ according to the measures on $\Delta_{|\sigma|}$ and set $\mathbf{Y}_{\{i_1,\ldots,i_n\}}=\rho(\{\xi_\sigma\}_{\sigma\subseteq\{i_1,\ldots,i_n\}})$. Consider some hypergraph event $(S,\{B_e\})$. The event $B_{S,\{B_e\},\mathbf{Y}}$ is precisely the intersection of the events $B_{e,\{B_e\},\mathbf{Y}}$, and is therefore the image of $B_{S,\{B_e\},M^*}$, and so has probability $\mathbb{P}(B_{S,\{B_e\},M^*})=\mathbb{P}(B_{S,\{B_e\},\mathbf{X}})$. Therefore $\mathbf{Y}$ has the same distribution as $\mathbf{X}$.
\end{proof}

\printbibliography

\end{document}